\documentclass[preprint,12pt,3p]{elsarticle}

\usepackage{amssymb}
\usepackage{amsmath,amsthm}
\usepackage{mathrsfs}
\usepackage{hyperref}
\usepackage{cleveref}
\usepackage[all,cmtip]{xy}
\usepackage{epsf, graphicx}
\usepackage{latexsym,amsfonts,amsbsy,amssymb}
\usepackage{geometry}
\usepackage{titletoc}
\usepackage{color,xcolor}
\usepackage{rotating}
\usepackage{algorithm}
\usepackage{algorithmic}
\usepackage{amscd}

\makeatletter

\@addtoreset{equation}{section} \makeatother
\newtheorem{theorem}{Theorem}[section]
\newtheorem{lemma}[theorem]{Lemma}

\newtheorem{remark}[theorem]{Remark}

\newtheorem{definition}[theorem]{Definition}
\newtheorem{proposition}[theorem]{Proposition}
\newtheorem{guess}[theorem]{Guess}

\def\Res{{\rm Res}}

\def\Inf{{\rm Inf}}

\def\Aut{{\rm Aut}}

\def\Hom{{\rm Hom}}

\def\End{{\rm End}}

\def\O{\mathcal{O}}

\def\F{\mathcal{F}}

\def\P{\mathcal{P}}
\def\Z{\mathbb{Z}}

\makeatletter
\def\ps@pprintTitle{%
\let\@oddhead\@empty
\let\@evenhead\@empty
\def\@oddfoot{\reset@font\hfil\thepage\hfil}
\let\@evenfoot\@oddfoot
}
\makeatother

\begin{document}

\begin{frontmatter}
	
	\title{Fusion-stable endosplit $p$-permutation resolutions}

	\author{Xin Huang}
	
	
	\begin{abstract}
Let $k$ be a field of characteristic $p>0$, $\mathcal{F}$ a saturated fusion system over a finite $p$-group $P$, and $V$ an indecomposable capped endopermutation $kP$-module. Let $D_k^\Omega(P)$ be the subgroup of the Dade group $D_k(P)$ generated by all the relative syzygies. It is known that $V$ has an endosplit $p$-permutation resolution if and only if the Dade class $[V]$ belongs to $D_k^\Omega(P)$. We show that the resolution can be chosen to be $\mathcal{F}$-stable if and only if $V$ is $\mathcal{F}$-stable. As an application, we prove the following folklore result: if two blocks of finite groups are Morita equivalent via a bimodule with an endopermutation $kP$-source $V$ such that $[V]\in D_k^\Omega(P)$, then these two blocks are splendidly Rickard equivalent.
	\end{abstract}
	
	
\end{frontmatter}


\section{Introduction}\label{s1}

Throughout this paper, $p$ is a prime and $k$ is a field of characteristic $p$. Let $P$ be a finite $p$-group. We denote by $D_k(P)$ the Dade group of $P$ over $k$ (see e.g. \cite[\S3]{Th07}). Let $X$ be a finite $P$-set and let $kX$ be the corresponding permutation $kP$-module. Let $\Omega_X(k)$ be the kernel of the augmentation homomorphism $kX\to k$ mapping every basis element in $X$ to 1. The $kP$-module $\Omega_X(k)$ is called a {\it relative syzygy} of $k$. By a result due to Alperin \cite{Alp01}, $\Omega_X(k)$ is an endopermutation $kP$-module.
Let $D_k^\Omega(P)$ be the subgroup of $D_k(P)$ generated by all the relative syzygies $\Omega_X(k)$, where $X$ runs over all non-empty finite $P$-sets. Equivalently, $D_k^\Omega(P)$ is generated by the classes $[\Omega_{P/Q}(k)]$, where $Q$ runs over the proper subgroups of $P$; see \cite[Lemma 5.2.3]{Bouc:relative syzygies}. According to the classification of endopermutation modules, if there are no subgroups $R\unlhd Q\leq P$ such that $Q/R\cong Q_8$ (the quaternion group of order $8$), then $D_k(P)=D_k^\Omega(P)$; see \cite[Theorem 9.5]{Bouc:The Dade group} or \cite[Theorems 13.2 and 13.3]{Th07}. It is also known that an endopermutation $kP$-module $V$ has an endosplit $p$-permutation resolution if and only if its class $[V]$ lies in $D_k^\Omega(P)$; see \cite[Theorem 14.3]{Th07}. We refer to \cite[Definition 2.2]{AKO} or \cite[Definition 8.1.6]{Lin18b} for the definition of a saturated fusion system. The main result of this paper is the following:

\begin{theorem}\label{theo:main}
Let $P$ be a finite $p$-group, $\F$ a fusion system over $P$, and $V$ an indecomposable capped endopermutation $kP$-module with $[V]\in D_k^\Omega(P)$. If $V$ has an $\F$-stable endosplit $p$-permutation resolution, then $V$ is $\F$-stable. Conversely, if $V$ is $\F$-stable and $\F$ is saturated, then $V$ has an $\F$-stable endosplit $p$-permutation resolution.
\end{theorem}

This will be proved in Section \ref{s3}, after reviewing some necessary background on endosplit $p$-permutation resolutions, fusion stability, and $p$-biset functors in Section \ref{s2:Preliminaries}. If $P$ is abelian, then $D_k^\Omega(P)=D_k(P)$ and Theorem \ref{theo:main} admits a much simpler proof, due to Linckelmann, which we present in Section \ref{s4}. By \cite[Theorems 9.11.2, 9.11.5 (i)]{Lin18b}, Theorem \ref{theo:main} implies the following folklore result:

\begin{theorem}\label{theorem:folklore}
Let $G$ and $H$ be finite groups, $b$ a block idempotent of $kG$ and $c$ a block idempotent of $kH$. Assume that $k$ is a splitting field for all subgroups of $G\times H$. Assume that $M$ is a $kGb$-$kHc$-bimodule inducing a Morita equivalence between $kGb$ and $kHc$. Assume further that when regarded as a $k(G\times H)$-module, $M$ has an endopermutation $kP$-source $V$ such that the class $[V]\in D_k^\Omega(P)$. Then $kGb$ and $kHc$ are splendidly Rickard equivalent.
\end{theorem}

In light of Theorem \ref{theorem:folklore}, \cite[Proposition 1.5]{H25} and \cite[Theorem 1.1]{H26}, we make the following guess:

\begin{guess}\label{guess}
Let $\O$ be a complete discrete valuation ring of characteristic $0$ with residue field $k$. Let $G$ and $H$ be finite groups, $b$ a block idempotent of $\O G$ and $c$ a block idempotent of $\O H$. Assume that $k$ is a splitting field for all subgroups of $G\times H$. Let $\bar{b}$ and $\bar{c}$ be the images of $b$ and $c$ in $kG$ and $kH$, respectively. Assume that $M$ is a $kG\bar{b}$-$kH\bar{c}$-bimodule inducing a Morita equivalence between $kG\bar{b}$ and $kH\bar{c}$. Assume that as a $k(G\times H)$-module, $M$ has an endopermutation $kP$-source $V$. Then the following are equivalent:
\begin{enumerate}[{\rm (i)}]
	\item $[V]\in D_k^\Omega(P)$.
	\item  $b$ and $c$ are globally isotypic.
	\item $b$ and $c$ are $p$-permutation equivalent.
		\item $b$ and $c$ are splendidly Rickard equivalent.
\end{enumerate}

\end{guess}

By a {\it global isotypy}, we mean an isotypy in \cite[Definition 9.5.1]{Lin18b}.

\section{Preliminaries}\label{s2:Preliminaries}

\subsection{Endosplit $p$-permutation resolutions and fusion stability}

Let $P$ be a finite $p$-group, and let $V$ be a $kP$-module. If ${\rm End}_k(V)$ admits a $P$-stable $k$-basis under the conjugation action, then $V$ is called an {\it endopermutation} $kP$-module, as defined in \cite{Dade}. An endopermutation $kP$-module $V$ is called {\it capped} if it has an indecomposable direct summand with vertex $P$. Any such summand is called a {\it cap} of $V$, and all caps of $V$ are isomorphic; see e.g. \cite[\S3]{Th07}. Thus a capped endopermutation module need not itself be indecomposable. In particular, an indecomposable capped endopermutation $kP$-module is precisely an indecomposable endopermutation $kP$-module with vertex $P$. For a bounded complex $X$ of $kG$-modules, we write $X^*=\Hom_k(X,k)$ for the dual complex, with $(X^*)_i=\Hom_k(X_{-i},k)$ and the usual differential; see e.g. \cite[Definition 1.17.10]{Lin18b}. We write $X\simeq Y$ if the complexes $X$ and $Y$ are homotopy equivalent.

\begin{definition}[{\cite[\S7]{Rickard}}; see also {\cite[Definition 7.11.1]{Lin18b}}]
{\rm Let $G$ be a finite group and let $M$ be a finitely generated $kG$-module. An {\it endosplit $p$-permutation resolution of $M$} is a bounded complex $X$ of finitely generated $p$-permutation $kG$-modules with homology concentrated in degree $0$ such that $H_0(X)\cong M$ and such that the complex $X\otimes_k X^*$ is a split complex of $kG$-modules, with $G$ acting diagonally on the tensor product.}
\end{definition}

\begin{definition}[{\cite[Definitions 9.9.1, 9.9.2]{Lin18b}}]\label{def:fusion-stability}
{\rm Let $\F$ be a fusion system over a finite $p$-group $P$. If $Q\leq P$, $\varphi:Q\to P$ is a morphism in $\F$, and $U$ is a $kP$-module, then ${}_\varphi U$ denotes the $kQ$-module obtained from $U$ by restriction along $\varphi$, so that $q\cdot u=\varphi(q)u$ for $q\in Q$ and $u\in U$; for a complex $U$, the notation is applied termwise. Let $V$ be an endopermutation $kP$-module. We say that $V$ is {\it $\F$-stable} if for any subgroup $Q$ of $P$ and any morphism $\varphi:Q\to P$ in $\F$, the sets of isomorphism classes of indecomposable direct summands with vertex $Q$ of the $kQ$-modules ${\rm Res}_Q^P(V)$ and ${}_\varphi V$ are equal (including the possibility that both sets may be empty). This is equivalent to requiring that ${\rm Res}_Q^P(V)\oplus {}_\varphi V$ is an endopermutation $kQ$-module; see \cite[Corollary 6.12]{Dade}. Assume that $V$ has an endosplit $p$-permutation resolution $Y$. Following \cite[Definition 9.9.2]{Lin18b}, we say that $Y$ is {\it $\F$-stable} if for any subgroup $Q$ of $P$ and any morphism $\varphi:Q\to P$ in $\F$, the sets of isomorphism classes of indecomposable direct summands with vertex $Q$ of the complexes $\Res_Q^P(Y)$ and ${}_\varphi Y$ are equal.}
\end{definition}

\subsection{Biset functors and fusion-stable elements}

\begin{definition}[{\cite[Chapters 2 and 3]{Bouc10}}]\label{def:p-biset-functor}
{\rm Let $G$ and $H$ be finite $p$-groups. An $(H,G)$-biset $U$ is a finite set equipped with a left action of $H$ and a right action of $G$ such that the two actions commute. Denote by $B(H,G)$ the Grothendieck group of the commutative monoid of isomorphism classes of finite $(H,G)$-bisets, where addition is induced by disjoint union. The {\it $p$-biset category} $\mathcal{B}_p$ has the finite $p$-groups as its objects and
$$
 \Hom_{\mathcal{B}_p}(G,H)=B(H,G).
$$
If $U$ is an $(H,G)$-biset and $W$ is a $(K,H)$-biset, let $\sim$ be the equivalence relation on $W\times U$ generated by
$$
 (wh,u)\sim(w,hu)\qquad(w\in W,\ u\in U,\ h\in H).
$$
The {\it balanced product} of $W$ and $U$ over $H$ is the quotient
$$
 W\times_H U=(W\times U)/{\sim}.
$$
It is a $(K,G)$-biset, and composition in $\mathcal{B}_p$ is defined by
$$
 [W]\circ[U]=[W\times_H U].
$$
The identity morphism of the object $G$ is represented by the regular $(G,G)$-biset $G$. A {\it $p$-biset functor} is an additive functor $M$ from $\mathcal{B}_p$ to the category  $\mathbf{Ab}$ of abelian groups. A morphism of $p$-biset functors is a natural transformation, and a sequence of $p$-biset functors is {\it exact} if its evaluation at every finite $p$-group is exact.}
\end{definition}

The equality $\Hom_{\mathcal{B}_p}(G,H)=B(H,G)$ fixes the direction in which a biset is used: an $(H,G)$-biset $U$ represents a morphism from $G$ to $H$. In other words, the group acting on the right is the source, while the group acting on the left is the target. Consequently, a $p$-biset functor $M$ sends $U$ to a homomorphism
$$
 M(U):M(G)\longrightarrow M(H).
$$
For example, if $H\leq G$, then the $(H,G)$-biset $G$ induces the restriction map $\Res_H^G:M(G)\to M(H)$, whereas the $(G,H)$-biset $G$ induces the induction map $\operatorname{Ind}_H^G:M(H)\to M(G)$.

Let $Q$ and $P$ be finite $p$-groups and let $\varphi:Q\to P$ be an injective homomorphism. Denote by ${}_\varphi P$ the $(Q,P)$-biset whose underlying set is $P$ and whose action is given by $q\cdot x\cdot u=\varphi(q)xu$. For a $p$-biset functor $M$, we write
$$
 \varphi^*=M({}_\varphi P):M(P)\longrightarrow M(Q).
$$
If $\varphi$ is the inclusion of a subgroup $Q\leq P$, this map is the usual restriction $\Res_Q^P$. If $\varphi$ is an isomorphism, $\varphi^*$ is a transport of structure. In general, $\varphi^*$ is the restriction to $\varphi(Q)$ followed by the transport of structure from $\varphi(Q)$ to $Q$.

\begin{definition}[{\cite[Definition 2.2]{RY18}}]\label{def:stable-elements}
{\rm Let $\F$ be a fusion system over a finite $p$-group $P$, and $M$ a $p$-biset functor. An element $a\in M(P)$ is called {\it $\F$-stable} if
$$
 \Res_Q^P(a)=\varphi^*(a)
$$
for any subgroup $Q\leq P$ and any $\varphi\in\Hom_\F(Q,P)$. The subgroup of $\F$-stable elements is denoted by
$$
 M(\F)=\{a\in M(P)\mid \Res_Q^P(a)=\varphi^*(a)\text{ for all }Q\leq P\text{ and }\varphi\in\Hom_\F(Q,P)\}.
$$}
\end{definition}

Let $D_k(P,\F)$ be the subgroup of $D_k(P)$ consisting of those classes whose caps are $\F$-stable in the sense of Definition \ref{def:fusion-stability}; see \cite[Definition 3.3 and Proposition 3.4]{LM09}. We write
$$
 D_k^\Omega(P,\F)=D_k^\Omega(P)\cap D_k(P,\F).
$$
For the $p$-biset functor $D_k^\Omega$, this intersection is precisely its subgroup of $\F$-stable elements in the sense of Definition \ref{def:stable-elements}. Let
$$
 \Z_{(p)}=\left\{\frac{a}{b}\in\mathbb{Q}\mathrel{\big|}a,b\in\Z,\ p\nmid b\right\}
$$
be the localisation of $\Z$ at the prime ideal $(p)$. For an abelian group $A$, write $A_{(p)}=\Z_{(p)}\otimes_\Z A$.

\begin{proposition}[{\cite[Proposition 2.6]{RY18}}]\label{prop:stable-elements}
Let $\F$ be a saturated fusion system over $P$. If
$$
 0\longrightarrow M_1\longrightarrow M_2\longrightarrow M_3\longrightarrow 0
$$
is an exact sequence of $p$-biset functors, then
$$
 0\longrightarrow M_1(\F)_{(p)}\longrightarrow M_2(\F)_{(p)}\longrightarrow M_3(\F)_{(p)}\longrightarrow 0
$$
is exact.
\end{proposition}

This proposition is the only place where saturation of $\F$ is used. In particular, it allows us to pass from an exact sequence of $p$-biset functors to an exact sequence of their $\F$-stable parts after localisation at $p$.

\section{Proof of Theorem \ref{theo:main}}\label{s3}

Throughout this section, $P$ denotes a finite $p$-group and $\F$ denotes a fusion system over $P$. We first review the relation between relative syzygies and integer-valued superclass functions.

\subsection{Superclass functions and relative syzygies}

\begin{definition}[{\cite[Definition 1.4]{Bouc04}}]\label{def:superclass}
{\rm Denote by $C(P)$ the abelian group of all functions from the set of subgroups of $P$ to $\Z$ which are constant on the $P$-conjugacy classes.  If $X$ is a finite $P$-set, define $\omega_X\in C(P)$ by
$$
 \omega_X(Q)=
 \begin{cases}
  1 & \text{if }X^Q\neq\varnothing,\\
  0 & \text{if }X^Q=\varnothing.
 \end{cases}
$$
In particular, for $R\leq P$, we write $\omega_{P/R}$ for the function associated with the transitive $P$-set $P/R$.}
\end{definition}

Thus $\omega_{P/R}(Q)=1$ if and only if $Q$ is contained in a $P$-conjugate of $R$.  If one orders the $P$-conjugacy classes of subgroups by their orders, the matrix
$$
 \big(\omega_{P/R}(Q)\big)_{Q,R}
$$
is triangular with diagonal entries equal to $1$.  Therefore, we see the following lemma:

\begin{lemma}[Bouc {\cite[Lemma 2.2]{Bouc04}}]\label{lem:omega-basis}
Let $\mathscr{S}$ be a set of representatives of the $P$-conjugacy classes of subgroups of $P$.  Then
$$
 \{\omega_{P/R}\mid R\in\mathscr{S}\}
$$
is a $\Z$-basis of $C(P)$.
\end{lemma}

\begin{definition}[Bouc--Yal\c{c}\i n {\cite[Definition 3.1]{BY07}}]\label{def:borel-smith}
{\rm A function $f\in C(P)$ is called a {\it Borel--Smith function} if the following conditions hold:
\begin{enumerate}
\item if $p$ is odd and $S\trianglelefteq T\leq P$ with $T/S\cong C_p$, then $f(S)-f(T)$ is even;
\item if $S\trianglelefteq T\leq P$ with $T/S\cong C_p\times C_p$, then
 \begin{equation}\label{eq:BS-linear}
  f(S)-f(T)=\sum_{S<U<T}\big(f(U)-f(T)\big);
 \end{equation}
\item if $p=2$ and $S\trianglelefteq T\trianglelefteq N\leq N_P(S)$ with $T/S\cong C_2$, then $f(S)-f(T)$ is even when $N/S\cong C_4$, and is divisible by $4$ when $N/S\cong Q_8$.
\end{enumerate}
The subgroup of all Borel--Smith functions in $C(P)$ is denoted by $C_b(P)$.}
\end{definition}

\begin{proposition}[Bouc--Yal\c{c}\i n {\cite[Theorem 1.2]{BY07}}]\label{prop:BY}
There is an exact sequence of $p$-biset functors
\begin{equation}\label{eq:BY}
 0\longrightarrow C_b\longrightarrow C
 \stackrel{\Psi}{\longrightarrow}D_k^\Omega\longrightarrow 0,
\end{equation}
where
$$
 \Psi_P(\omega_{P/R})=
 \begin{cases}
  [\Omega_{P/R}(k)]&\text{if }R<P,\\
  0&\text{if }R=P.
 \end{cases}
$$
\end{proposition}

For the $p$-biset functor $C$, the map $\varphi^*$ introduced before Definition \ref{def:stable-elements} is given by
$$
 (\varphi^*f)(R)=f(\varphi(R))
$$
for $f\in C(P)$ and $R\leq Q$. Thus the stable-element subgroup from Definition \ref{def:stable-elements} is
$$
 C(\F)=\{f\in C(P)\mid \Res_Q^P(f)=\varphi^*f
 \text{ for all }Q\leq P\text{ and }\varphi\in\Hom_\F(Q,P)\}.
$$
In Proposition \ref{prop:BY}, $\Psi$ is a natural transformation between $p$-biset functors, hence $\Psi_P$ maps $C(\F)$ to $D_k^\Omega(P,\F)$.  The main point of the next subsection is that the restricted map $C(\F)\to D_k^\Omega(P,\F)$ is surjective.

\subsection{A fusion-stable lifting}

Let
$$
 \Z[1/p]=\left\{\frac{a}{p^n}\in\mathbb{Q}\mathrel{\big|}a\in\Z,\ n\geq 0\right\}
$$
be the localisation of $\Z$ obtained by inverting $p$. For an abelian group $A$, write
$$
 A[1/p]=\Z[1/p]\otimes_\Z A.
$$

\begin{lemma}\label{lem:p-local}
Assume that $\F$ is saturated. With the notation of Proposition \ref{prop:BY}, the homomorphism
$$
 \Psi_P:C(\F)_{(p)}\longrightarrow D_k^\Omega(P,\F)_{(p)}
$$
is surjective.
\end{lemma}

\begin{proof}
Apply Proposition \ref{prop:stable-elements} to the exact sequence \eqref{eq:BY}. The resulting sequence on $\F$-stable parts is exact after localisation at $p$, and hence its last map $\Psi_P:C(\F)_{(p)}\to D_k^\Omega(P,\F)_{(p)}$ is surjective.
\end{proof}

We next prove the corresponding statement after inverting $p$. Let $C_{\rm lin}(P)$ be the subgroup of $C(P)$ consisting of the functions which satisfy \eqref{eq:BS-linear} whenever $S\trianglelefteq T\leq P$ and $T/S\cong C_p\times C_p$. It follows directly from the Borel--Smith conditions that
\begin{equation}\label{eq:Cb-Cl}
 \begin{aligned}
 C_b(P)[1/2]&=C_{\rm lin}(P)[1/2] &&\text{if }p=2,\\
 2C_{\rm lin}(P)&\subseteq C_b(P) &&\text{if }p\text{ is odd}.
 \end{aligned}
\end{equation}
The following elementary projection will be useful.

\begin{lemma}\label{lem:projection}
There is a homomorphism
$$
 \pi_P:C(P)[1/p]\longrightarrow C_{\rm lin}(P)[1/p]
$$
with the following properties:
\begin{enumerate}
\item $\pi_P$ is the identity on $C_{\rm lin}(P)[1/p]$;
\item the homomorphisms $\pi_P$ commute with restriction and transport by injective group homomorphisms.
\end{enumerate}
\end{lemma}

\begin{proof}
For a nontrivial cyclic subgroup $C$ of $P$, denote by $C^-$ its unique subgroup of index $p$.  For $u\in C(P)[1/p]$ and $H\leq P$, set
\begin{equation}\label{eq:projection}
 (\pi_Pu)(H)=\frac{1}{|H|}\left(
 u(1)+\sum_{\substack{1<C\leq H\\ C\text{ cyclic}}}
 \big(|C|u(C)-|C^-|u(C^-)\big)\right).
\end{equation}
The sum in \eqref{eq:projection} is over all cyclic subgroups of $H$, not only over their conjugacy classes.  Since $|H|$ is a power of $p$, the right side belongs to $\Z[1/p]$. We verify first that $\pi_Pu$ satisfies \eqref{eq:BS-linear}.  Define a rational-valued function $\chi_u$ on the elements of $P$ by
$$
 \chi_u(1)=u(1)~~~{\rm and}~~~
 \chi_u(x)=\frac{p\,u(\langle x\rangle)
 -u(\langle x^p\rangle)}{p-1}\quad (x\neq 1).
$$
By grouping the elements of $H$ according to the cyclic subgroups which they generate, we obtain
$$
 (\pi_Pu)(H)=\frac{1}{|H|}\sum_{x\in H}\chi_u(x).
$$
Indeed, for each nontrivial cyclic subgroup $C\leq H$, the elements $x\in H$ satisfying $\langle x\rangle=C$ are precisely the generators of $C$, of which there are $|C|-|C^-|=(p-1)|C^-|$. For every such $x$, we have $\langle x^p\rangle=C^-$. Therefore,
$$
 \begin{aligned}
 \sum_{x\in H}\chi_u(x)
 &=u(1)+\sum_{\substack{1<C\leq H\\C\text{ cyclic}}}
 (|C|-|C^-|)\frac{p\,u(C)-u(C^-)}{p-1}\\
 &=u(1)+\sum_{\substack{1<C\leq H\\C\text{ cyclic}}}
 \bigl(|C|u(C)-|C^-|u(C^-)\bigr),
 \end{aligned}
$$
which is $|H|(\pi_Pu)(H)$ by \eqref{eq:projection}. To verify \eqref{eq:BS-linear}, suppose that $S\trianglelefteq T\leq P$ and $T/S\cong C_p\times C_p$, and let $S<U_0,\ldots,U_p<T$ be the intermediate subgroups.  Every element of $T\setminus S$ belongs to exactly one $U_i$, while every element of $S$ belongs to all the $U_i$.  Consequently,
$$
 \sum_{i=0}^{p}\sum_{x\in U_i}\chi_u(x)
 =p\sum_{x\in S}\chi_u(x)+\sum_{x\in T}\chi_u(x).
$$
Let $f=\pi_Pu$. Since $|U_i|=p|S|$ for every $i$ and $|T|=p^2|S|$, dividing the preceding equality by $p|S|$, we obtain
$$
 \begin{aligned}
 \sum_{i=0}^{p}f(U_i)
 &=\sum_{i=0}^{p}\frac{1}{|U_i|}\sum_{x\in U_i}\chi_u(x)\\
 &=\frac{1}{|S|}\sum_{x\in S}\chi_u(x)+\frac{p}{|T|}\sum_{x\in T}\chi_u(x)\\
 &=f(S)+pf(T).
 \end{aligned}
$$
Equivalently,
$$
 f(S)-f(T)=\sum_{i=0}^{p}\bigl(f(U_i)-f(T)\bigr),
$$
which is precisely \eqref{eq:BS-linear}. Hence $\pi_Pu\in C_{\rm lin}(P)[1/p]$.

If $H$ is cyclic, the sum in \eqref{eq:projection} telescopes, and hence $(\pi_Pu)(H)=u(H)$.  On the other hand, a function satisfying \eqref{eq:BS-linear} is determined by its values on the cyclic subgroups. Indeed, if $H$ is noncyclic, then $H$ has a normal subgroup $S$ such that $H/S\cong C_p\times C_p$, and \eqref{eq:BS-linear} determines the value at $H$ from the values at proper subgroups of $H$.  Induction on $|H|$ now shows that $\pi_P$ is the identity on $C_{\rm lin}(P)[1/p]$. Finally, formula \eqref{eq:projection} is unchanged under an injective group homomorphism, and hence the homomorphisms $\pi_P$ have the property asserted in (ii).
\end{proof}

\begin{lemma}\label{lem:away-p}
The homomorphism
$$
 \Psi_P:C(\F)[1/p]\longrightarrow D_k^\Omega(P,\F)[1/p]
$$
is surjective.
\end{lemma}

\begin{proof}
Let $d\in D_k^\Omega(P,\F)$ and choose $f\in C(P)$ such that $\Psi_P(f)=d$.  For $Q\leq P$ and $\varphi\in\Hom_\F(Q,P)$, set
$$
 \delta_\varphi(f)=\Res_Q^P(f)-\varphi^*f.
$$
Since $d$ is $\F$-stable, Proposition \ref{prop:BY} implies
\begin{equation}\label{eq:defect-Cb}
 \delta_\varphi(f)\in C_b(Q).
\end{equation}
Notice also that $\delta_\varphi(f)(1)=0$. Let $\pi_P$ be the homomorphism defined in Lemma \ref{lem:projection}. Suppose first that $p=2$.  Define
$$
 b=\pi_P(f)\in C_{\rm lin}(P)[1/2]=C_b(P)[1/2].
$$
By property (ii) of Lemma \ref{lem:projection}, we have
$$
 \begin{aligned}
 \delta_\varphi(b)
 &=\Res_Q^P\bigl(\pi_P(f)\bigr)-\varphi^*\bigl(\pi_P(f)\bigr)\\
 &=\pi_Q\bigl(\Res_Q^P(f)\bigr)-\pi_Q(\varphi^*f)\\
 &=\pi_Q\bigl(\delta_\varphi(f)\bigr).
 \end{aligned}
$$
On the other hand, \eqref{eq:defect-Cb} and \eqref{eq:Cb-Cl} imply
$$
 \delta_\varphi(f)\in C_b(Q)\subseteq C_b(Q)[1/2]=C_{\rm lin}(Q)[1/2].
$$
Therefore, by property (i) of Lemma \ref{lem:projection}, we have
$$
 \pi_Q\bigl(\delta_\varphi(f)\bigr)=\delta_\varphi(f).
$$
Consequently, $\delta_\varphi(f-b)=0$. Since $Q$ and $\varphi$ were arbitrary, $f-b$ is $\F$-stable. Moreover, $b\in C_b(P)[1/2]$, so Proposition \ref{prop:BY} implies that $\Psi_P(b)=0$. Hence $\Psi_P(f-b)=d$ in $D_k^\Omega(P)[1/2]$.

Now suppose that $p$ is odd.  We claim that the reduction of $f$ modulo $2$ is $\F$-stable.  Indeed, let $R\leq Q$.  Choose a subnormal series from $1$ to $R$ whose factors have order $p$.  The first Borel--Smith condition, applied to $\delta_\varphi(f)$, and the equality $\delta_\varphi(f)(1)=0$ show that $\delta_\varphi(f)(R)$ is even.  Define $\epsilon\in C(P)$ by requiring
$$
 \epsilon(R)\in\{0,1\}\qquad\text{and}\qquad
 \epsilon(R)\equiv f(R)\pmod 2
$$
for all $R\leq P$.  Then $\epsilon\in C(\F)$ and $f-\epsilon\in 2C(P)$.  Set
$$
 b=\pi_P(f-\epsilon).
$$
Since $f-\epsilon\in 2C(P)$ and $\pi_P$ is a homomorphism with image in $C_{\rm lin}(P)[1/p]$, we have $b\in 2C_{\rm lin}(P)[1/p]$. Hence, by \eqref{eq:Cb-Cl}, we have
$$
 b\in C_b(P)[1/p].
$$
Since $\epsilon\in C(\F)$, we have $\delta_\varphi(\epsilon)=0$. Therefore, by property (ii) of Lemma \ref{lem:projection}, we have
$$
 \begin{aligned}
 \delta_\varphi(b)
 &=\delta_\varphi\bigl(\pi_P(f-\epsilon)\bigr)\\
 &=\pi_Q\bigl(\delta_\varphi(f-\epsilon)\bigr)\\
 &=\pi_Q\bigl(\delta_\varphi(f)-\delta_\varphi(\epsilon)\bigr)\\
 &=\pi_Q\bigl(\delta_\varphi(f)\bigr).
 \end{aligned}
$$
By \eqref{eq:defect-Cb}, $\delta_\varphi(f)$ belongs to $C_b(Q)$. Every Borel--Smith function satisfies \eqref{eq:BS-linear}, so $C_b(Q)\subseteq C_{\rm lin}(Q)$. Property (i) of Lemma \ref{lem:projection} now implies that
$$
 \pi_Q\bigl(\delta_\varphi(f)\bigr)=\delta_\varphi(f).
$$
Consequently, $\delta_\varphi(f-b)=0$. Since $Q$ and $\varphi$ were arbitrary, $f-b$ is $\F$-stable. Moreover, $b\in C_b(P)[1/p]$, so Proposition \ref{prop:BY} implies that $\Psi_P(b)=0$. Hence $\Psi_P(f-b)=d$ in $D_k^\Omega(P)[1/p]$.
\end{proof}

\begin{proposition}\label{prop:stable-lift}
Assume that $\F$ is saturated. Then the homomorphism
$$
 \Psi_P:C(\F)\longrightarrow D_k^\Omega(P,\F)
$$
is surjective.
\end{proposition}

\begin{proof}
Let $A$ be its cokernel. By Lemma \ref{lem:p-local}, we have $A_{(p)}=0$, while by Lemma \ref{lem:away-p}, we have $A[1/p]=0$. Hence every element of $A$ is annihilated both by an integer prime to $p$ and by a power of $p$. These two integers are coprime, and B\'ezout's identity implies that the element is zero. Thus $A=0$.
\end{proof}

\begin{remark}\label{rem:no-canonicity}
{\rm No uniqueness is asserted in Proposition \ref{prop:stable-lift}.  Two $\F$-stable lifts of the same element of $D_k^\Omega(P,\F)$ may differ by an $\F$-stable Borel--Smith function.}
\end{remark}

\subsection{Construction of the resolution}

We now relate the functions considered above to endosplit $p$-permutation complexes. Let $X$ be a bounded complex of $p$-permutation $kP$-modules such that $X\otimes_k X^*$ is split and such that the unique nonzero homology module $M$ of $X$ is a capped endopermutation module. We do not require, for the moment, that this homology occurs in degree $0$. For $Q\leq P$, the Brauer construction $X(Q)$ has homology concentrated in exactly one degree. Indeed, since $X\otimes_k X^*$ is split, we have $X\otimes_k X^*\simeq\End_k(M)[0]$. Applying the Brauer construction and using its compatibility with tensor products and duals, we obtain $X(Q)\otimes_k X(Q)^*\simeq\End_k(M)(Q)[0]$. Since $M$ is capped, $\End_k(M)$ has the trivial module as a direct summand, and hence $\End_k(M)(Q)\neq0$. Thus $X(Q)$ is not contractible, while the K\"unneth formula shows that its homology is concentrated in at most one degree. Denote this unique degree by $h_X(Q)$. If $Q$ and $R$ are conjugate in $P$, then $X(Q)$ and $X(R)$ are isomorphic up to transport of structure, so $h_X(Q)=h_X(R)$. Therefore $h_X\in C(P)$.  We refer to $h_X$ as the {\it homological mark} of $X$.  We collect the properties needed below.

\begin{lemma}\label{lem:h-marks}
Let $X$ and $Y$ be bounded complexes of $p$-permutation $kP$-modules. Assume that $X\otimes_k X^*$ and $Y\otimes_k Y^*$ are split and that each of $X$ and $Y$ has a capped endopermutation $kP$-module as its unique nonzero homology module. The following hold:
\begin{enumerate}
\item $h_{X\otimes_k Y}=h_X+h_Y$ and $h_{X^*}=-h_X$.
\item If $Q\leq P$ and $\varphi:Q\to P$ is injective, then for $R\leq Q$,
 $$
  h_{\Res_Q^P(X)}(R)=h_X(R),~~~{\rm and}~~~
  h_{{}_\varphi X}(R)=h_X(\varphi(R)).
 $$
\item Suppose that $h_X=h_Y$.  Then the noncontractible indecomposable direct summands with vertex $P$ of $X$ and $Y$ have the same isomorphism class.
\end{enumerate}
\end{lemma}

\begin{proof}
The first two statements follow from the K\"unneth formula and from the fact that the Brauer construction commutes with tensor products of $p$-permutation modules. For (iii), the equality $h_X=h_Y$ implies that $h_X(1)=h_Y(1)$, so the unique nonzero homology modules of $X$ and $Y$ occur in the same degree. After applying the same shift to both complexes, we may assume that this degree is $0$, so that $X$ and $Y$ are endosplit $p$-permutation resolutions. Apply \cite[Theorem 3.13]{Mil26} with $V=\bigoplus_{R<P}k[P/R]$. Then the homological mark homomorphism is injective. Take noncontractible indecomposable summands $X_0$ and $Y_0$ with vertex $P$. Their Brauer constructions are noncontractible at every subgroup of $P$, and hence $h_{X_0}=h_X$ and $h_{Y_0}=h_Y$. Therefore, by the injectivity of the homological mark homomorphism, we have $X_0\cong Y_0$. The correspondence between indecomposable summands of the homology and noncontractible indecomposable summands of an endosplit resolution is vertex preserving; see \cite[Theorem 7.11.2]{Lin18b}. This proves (iii).
\end{proof}

\begin{proposition}\label{prop:relative-syzygy-resolution}
Let $S$ be a nonempty finite $P$-set, and let
$$
 C_S=\big(0\longrightarrow kS
 \stackrel{\varepsilon_S}{\longrightarrow}k\longrightarrow0\big),
$$
where $kS$ is in degree $1$, $k$ is in degree $0$, and $\varepsilon_S$ is the augmentation map. Then $C_S[-1]$ is an endosplit $p$-permutation resolution of $\Omega_S(k)$. In particular, if we let $X_Q=C_{P/Q}$ for any $Q<P$, then
\begin{equation}\label{eq:h-standard}
 h_{X_Q}=\omega_{P/Q}.
\end{equation}
\end{proposition}

\begin{proof}
Since $S$ is nonempty, $\varepsilon_S$ is surjective, and hence the only nonzero homology of $C_S$ is $\Omega_S(k)$ in degree $1$. It remains to prove that $C_S\otimes_k C_S^*$ is split. Let $\sigma=\sum_{x\in S}x\in kS$. After identifying $(kS)^*$ with $kS$ via the dual of the permutation basis, we may write $C_S\otimes_k C_S^*$ as
$$
 0\longrightarrow kS\stackrel{d_1}{\longrightarrow}
 (kS\otimes_k kS)\oplus k
 \stackrel{d_0}{\longrightarrow}kS\longrightarrow0,
$$
where for any $x,y\in S$,
$$
 d_1(x)=(x\otimes\sigma,-1),\qquad
 d_0(x\otimes y,0)=y,\qquad d_0(0,1)=\sigma.
$$
Define $kP$-homomorphisms
$$
 r:(kS\otimes_k kS)\oplus k\longrightarrow kS,
 \qquad r(x\otimes y,0)=\delta_{x,y}x,\quad r(0,1)=0,
$$
where $\delta_{x,y}$ denotes the Kronecker delta, and
$$
 s:kS\longrightarrow(kS\otimes_k kS)\oplus k,
 \qquad s(x)=(x\otimes x,0).
$$
Both maps are $P$-equivariant, and $r d_1=\operatorname{id}_{kS}$ and $d_0s=\operatorname{id}_{kS}$. Thus $d_1$ is split injective and $d_0$ is split surjective, so $C_S\otimes_k C_S^*$ is split. Since $(C_S[-1])\otimes_k(C_S[-1])^*\cong C_S\otimes_k C_S^*$, the shifted complex is an endosplit $p$-permutation resolution of $\Omega_S(k)$.

Now let $Q<P$ and set $X_Q=C_{P/Q}$. If $(P/Q)^R$ is empty, then $X_Q(R)$ is the complex $0\to k$ and has homology in degree $0$. If $(P/Q)^R$ is nonempty, then it has at least $p$ elements: if $R=1$, then $(P/Q)^R=P/Q$ and $|P/Q|=|P:Q|$ is divisible by $p$ because $Q<P$; if $R\neq1$, then orbit counting for the action of the $p$-group $R$ on $P/Q$ shows that $|(P/Q)^R|\equiv |P/Q|\equiv0\pmod p$. Thus $X_Q(R)$ has nonzero homology in degree $1$, and hence $h_{X_Q}=\omega_{P/Q}$.
\end{proof}

\begin{proposition}\label{prop:realise-f}
Let $f\in C(P)$ and assume that $f(1)=0$. There is an endosplit $p$-permutation complex $X(f)$ of $kP$-modules such that
$$
 h_{X(f)}=f
$$
and such that $H_0(X(f))$ is an endopermutation $kP$-module whose class in $D_k^\Omega(P)$ is $\Psi_P(f)$.
\end{proposition}

\begin{proof}
Let $\mathscr{S}$ be as in Lemma \ref{lem:omega-basis} with $P\in \mathscr{S}$, and write
$$
 f=\sum_{Q\in\mathscr{S}}n_Q\omega_{P/Q}.
$$
For an integer $n$, let $k[n]$ denote the complex with $k$ concentrated in degree $n$. Define
$$
 \begin{aligned}
 X(f)={}&k[n_P]
 \otimes_k
 \bigotimes_{\substack{Q\in\mathscr{S},\ Q<P\\n_Q>0}}
 X_Q^{\otimes n_Q}
 \otimes_k
 \bigotimes_{\substack{Q\in\mathscr{S},\ Q<P\\n_Q<0}}
 (X_Q^*)^{\otimes(-n_Q)}.
 \end{aligned}
$$
Tensor products, duals and shifts preserve the endosplit property.  Hence $X(f)$ is an endosplit $p$-permutation complex.  By Lemma \ref{lem:h-marks} (i) and \eqref{eq:h-standard}, we have $h_{X(f)}=f$.  In particular, the homology of $X(f)$ lies in degree
$$
 h_{X(f)}(1)=f(1)=0.
$$
The K\"unneth formula shows that the class of $H_0(X(f))$ is
$$
 \sum_{\substack{Q\in\mathscr{S}\\Q<P}}
  n_Q[\Omega_{P/Q}(k)].
$$
By the definition of $\Psi_P$ in Proposition \ref{prop:BY}, this sum is $\Psi_P(f)$.
\end{proof}

To verify the condition in Definition \ref{def:fusion-stability}, we must compare not only the noncontractible indecomposable summands, but also the contractible ones.  The following lemma shows that the latter can be made to agree after adding suitable split contractible complexes.

\begin{lemma}\label{lem:contractible-padding}
Let $Z$ be a bounded complex of $p$-permutation $kP$-modules, concentrated in degrees from $a$ to $b$. For a finite group $G$, a $kG$-module $W$, and an integer $i$, let
$$
 D_i(W)=\big(0\longrightarrow W\xrightarrow{\operatorname{id}_W}W\longrightarrow0\big)
$$
denote the complex whose two copies of $W$ are concentrated in degrees $i$ and $i-1$, respectively.  Thus $D_i(W)$ is contractible.  Let
$$
 Z^+=Z\oplus\bigoplus_{i=a+1}^{b}D_i(k).
$$
For every $Q\leq P$ and every injective homomorphism $\varphi:Q\to P$, the sets of isomorphism classes of contractible indecomposable direct summands with vertex $Q$ of $\Res_Q^P(Z^+)$ and ${}_\varphi Z^+$ are equal.
\end{lemma}

\begin{proof}
An elementary induction on the length of a contractible complex shows that every indecomposable contractible complex of $p$-permutation $kQ$-modules is of the form $D_i(W)$, where $W$ is an indecomposable $p$-permutation $kQ$-module. Indeed, in the largest nonzero degree, the contracting homotopy splits the differential, producing such a two-term direct summand; one then continues with the remaining shorter complex. Since $Q$ is a $p$-group, $W\cong k[Q/S]$ for some $S\leq Q$, and the vertex of $W$ is $S$. Hence $D_i(W)$ has vertex $Q$ if and only if $W\cong k$. Both restriction and transport of structure preserve the degree range $[a,b]$, and the added summands supply one copy of $D_i(k)$ for every $a<i\leq b$. Thus all possible contractible vertex-$Q$ classes occur on both sides, which proves the assertion.
\end{proof}

We are now ready to prove the main theorem.

\begin{proof}[Proof of Theorem \ref{theo:main}]
Suppose first that $V$ has an $\F$-stable endosplit $p$-permutation resolution $Z$. Let $Q\leq P$ and $\varphi\in\Hom_\F(Q,P)$. Restriction and transport of structure are exact and preserve the endosplit property. Hence $\Res_Q^P(Z)$ and ${}_\varphi Z$ are endosplit $p$-permutation resolutions of $\Res_Q^P(V)$ and ${}_\varphi V$, respectively, and
$$
 H_0\big(\Res_Q^P(Z)\big)\cong\Res_Q^P(V),\qquad
 H_0({}_\varphi Z)\cong{}_{\varphi}V.
$$
The functor $H_0$ preserves direct sums and sends isomorphic complexes to isomorphic modules. Moreover, the correspondence between the noncontractible indecomposable direct summands of an endosplit resolution and the indecomposable direct summands of its homology is vertex preserving; see \cite[Theorem 7.11.2]{Lin18b}. Since $Z$ is $\F$-stable, the noncontractible indecomposable direct summands with vertex $Q$ of $\Res_Q^P(Z)$ and ${}_\varphi Z$ have the same isomorphism classes. Applying $H_0$ therefore shows that the indecomposable direct summands with vertex $Q$ of $\Res_Q^P(V)$ and ${}_\varphi V$ have the same isomorphism classes. Thus $V$ is $\F$-stable.

Conversely, suppose that $V$ is $\F$-stable and that $\F$ is saturated. Then the class $d=[V]$ belongs to $D_k^\Omega(P,\F)$.  By Proposition \ref{prop:stable-lift}, there is an $f\in C(\F)$ such that $\Psi_P(f)=d$.  Since the constant function $\omega_{P/P}$ belongs to the kernel of $\Psi_P$, replacing $f$ by $f-f(1)\omega_{P/P}$ allows us to assume that $f(1)=0$.

Let $X=X(f)$ be the complex in Proposition \ref{prop:realise-f}, and let $M=H_0(X)$. The classes of $M$ and $V$ in $D_k(P)$ are equal. Since $V$ is an indecomposable capped endopermutation $kP$-module, it is its own cap. Hence the definition of equivalence in the Dade group implies that $V$ is isomorphic to a cap of $M$. By the direct summand correspondence for endosplit resolutions, $X$ has an indecomposable noncontractible direct summand $Y$ such that
$$
 H_0(Y)\cong V.
$$
Moreover, $Y$ is an endosplit $p$-permutation resolution of $V$.

For every $R\leq P$, the complex $Y(R)$ is noncontractible. Indeed, $\End_k(V)$ has the trivial module as a direct summand, and hence its Brauer construction at $R$ is nonzero. Applying the Brauer construction to the homotopy equivalence
$$
 Y\otimes_k Y^*\simeq \End_k(V)[0]
$$
shows that $Y(R)$ cannot be contractible. Since $Y(R)$ is a direct summand of $X(R)$, whose homology is concentrated in degree $f(R)$, it follows that
\begin{equation}\label{eq:hY=f}
 h_Y(R)=h_X(R)=f(R).
\end{equation}

Let $Q\leq P$ and let $\varphi\in\Hom_\F(Q,P)$.  For every $R\leq Q$, by Lemma \ref{lem:h-marks} (ii), equation \eqref{eq:hY=f}, and the $\F$-stability of $f$, we have
$$
 h_{\Res_Q^P(Y)}(R)=f(R)=f(\varphi(R))=h_{{}_\varphi Y}(R).
$$
Lemma \ref{lem:h-marks} (iii), applied over $Q$, now shows that the sets of isomorphism classes of noncontractible indecomposable direct summands with vertex $Q$ of $\Res_Q^P(Y)$ and ${}_\varphi Y$ are equal.

It remains only to take care of contractible direct summands, since these are also included in Definition \ref{def:fusion-stability}.  Apply Lemma \ref{lem:contractible-padding} to $Y$, and let $Y^+$ be the resulting complex.  Adding split contractible complexes does not change the homology or the endosplit property.  Lemma \ref{lem:contractible-padding} deals with the contractible vertex-$Q$ summands, while the preceding paragraph deals with the noncontractible ones.  Therefore $Y^+$ is an $\F$-stable endosplit $p$-permutation resolution of $V$.
\end{proof}

\section{Appendix: Alternative proof of Theorem \ref{theo:main} for abelian $p$-groups}\label{s4}

We present a direct/simpler proof of Theorem \ref{theo:main}, due to Linckelmann \cite[Theorem 7.11.9 (i)]{Lin18b}, for abelian $p$-groups $P$. In this case $D_k^\Omega(P)=D_k(P)$. The proof of \cite[Theorem 7.11.9 (i)]{Lin18b} seems slightly incomplete, but provides the idea of a proof. Let $P$ be a finite $p$-group. For $\alpha\in\Aut(P)$ and a $kP$-module (or a complex of $kP$-modules) $X$, write ${}^\alpha X={}_{\alpha^{-1}}X$. Thus the action of $u\in P$ on ${}^\alpha X$ is the action of $\alpha^{-1}(u)$ on $X$. If $E\leq\Aut(P)$, $X$ is called {\it $E$-stable} if
${}^\alpha X\cong X$ for any $\alpha\in E$.
\begin{lemma}\label{lem:abelian-normal-form}
Let $P$ be a finite abelian $p$-group. For $Q<P$, set
$$
 \omega_Q=\left[\Inf_{P/Q}^P\Omega_{P/Q}(k)\right]\in D_k(P).
$$
Then
\begin{equation}\label{eq:abelian-dade-group}
 D_k(P)\cong\bigoplus_{\substack{Q<P\\P/Q\text{ noncyclic}}}\Z\omega_Q\oplus\bigoplus_{\substack{Q<P,\ |P:Q|\geq3\\P/Q\text{ cyclic}}}(\Z/2\Z)\omega_Q.
\end{equation}
Let $E\leq\Aut(P)$ and let $M$ be an $E$-stable indecomposable endopermutation $kP$-module with vertex $P$. There is a unique family $\mathbf{n}=(n_Q)_{Q<P,\ |P:Q|\geq3}$, where $n_Q\in\Z$ if $P/Q$ is noncyclic and $n_Q\in\{0,1\}$ if $P/Q$ is cyclic, such that $M$ is isomorphic to the indecomposable direct summand with vertex $P$ of
\begin{equation}\label{eq:abelian-normal-form}
 T(\mathbf{n})=\bigotimes_{\substack{Q<P\\|P:Q|\geq3}}\Inf_{P/Q}^P\Omega_{P/Q}^{\,n_Q}(k).
\end{equation}
Here $\Omega_{P/Q}^{\,n}(k)$ denotes the $n$-th Heller translate of the trivial $k(P/Q)$-module, and in the cyclic case the exponent is read modulo $2$.
Moreover, $n_{\alpha(Q)}=n_Q$ for all $Q<P$ and all $\alpha\in E$.
\end{lemma}

\begin{proof}
The decomposition \eqref{eq:abelian-dade-group} and the normal form \eqref{eq:abelian-normal-form} follow from Dade's classification of the indecomposable endopermutation modules for abelian $p$-groups; see \cite[Theorem 12.5]{Dade2} or \cite[Theorem 5.2]{Th07}, or \cite[Theorem 7.8.1]{Lin18b}. For $\alpha\in\Aut(P)$, transport of structure commutes with inflation and Heller translates, and the isomorphism $P/Q\cong P/\alpha(Q)$ induced by $\alpha$ shows that ${}^\alpha\omega_Q=\omega_{\alpha(Q)}$. Since ${}^\alpha M\cong M$, the class $[M]$ is fixed by $\alpha$. The uniqueness of the coefficients in \eqref{eq:abelian-dade-group} therefore implies that $n_{\alpha(Q)}=n_Q$, with the equality interpreted in $\Z/2\Z$ when $P/Q$ is cyclic. Choosing the representative in $\{0,1\}$ gives the asserted equality of integers in that case.
\end{proof}

\begin{lemma}\label{lem:ambient-resolution}
Keep the notation of Lemma \ref{lem:abelian-normal-form}. The module $T(\mathbf{n})$ has an $E$-stable endosplit $p$-permutation resolution $C(\mathbf{n})$.
\end{lemma}

\begin{proof}
For a finite $p$-group $R$ and $n\in\Z$, let $B_R(n)$ be the endosplit $p$-permutation resolution of $\Omega_R^n(k)$ obtained by truncating a minimal projective resolution of $k$, shifting its homology to degree $0$, and using the dual complex when $n<0$; set $B_R(0)=k[0]$. These are the resolutions constructed in \cite[Proposition 7.11.7 and Corollary 7.11.8]{Lin18b}. If $f:R\to R'$ is a group isomorphism, the transport along $f$ of a minimal projective resolution of $k$ is again a minimal projective resolution of $k$. By the uniqueness of minimal projective resolutions, the transport of $B_R(n)$ is isomorphic to $B_{R'}(n)$ as a complex. Define
\begin{equation}\label{eq:ambient-resolution}
 C(\mathbf{n})=\bigotimes_{\substack{Q<P\\|P:Q|\geq3}}\Inf_{P/Q}^P B_{P/Q}(n_Q).
\end{equation}
By \cite[Proposition 7.11.6 (ii)]{Lin18b}, this is an endosplit $p$-permutation resolution of $T(\mathbf{n})$. For $\alpha\in E$, transport by $\alpha$ permutes the tensor factors in \eqref{eq:ambient-resolution} according to $Q\mapsto\alpha(Q)$. By Lemma \ref{lem:abelian-normal-form}, the corresponding exponents are equal. The isomorphisms between the transported complexes $B_{P/Q}(n_Q)$ and $B_{P/\alpha(Q)}(n_{\alpha(Q)})$, together with the permutation isomorphisms for tensor factors, therefore yield ${}^\alpha C(\mathbf{n})\cong C(\mathbf{n})$.
\end{proof}

\begin{proposition}[{\cite[Theorem 7.11.9 (i)]{Lin18b}}]\label{prop:abelian-stable-resolution}
Let $P$ be a finite abelian $p$-group, let $E\leq\Aut(P)$, and let $M$ be an $E$-stable indecomposable endopermutation $kP$-module with vertex $P$. Then $M$ has an $E$-stable endosplit $p$-permutation resolution.
\end{proposition}

\begin{proof}
Let $T=T(\mathbf{n})$ and $C=C(\mathbf{n})$ be as in Lemmas \ref{lem:abelian-normal-form} and \ref{lem:ambient-resolution}. Let $\End_k(C)$ denote the endomorphism complex of $C$, and write $K^b(kP)$ for the bounded homotopy category of finitely generated $kP$-modules. Since $\End_k(C)\cong C\otimes_k C^*$, this is a split complex of $kP$-modules. Consequently, taking $P$-fixed points commutes with taking its homology. Since $H_0(C)\cong T$, we obtain an algebra isomorphism
\begin{equation}\label{eq:homotopy-endomorphism-algebra}
 \begin{aligned}
 \End_{K^b(kP)}(C)&\cong H^0\!\left(\End_k(C)^P\right)\\
 &\cong H^0\!\left(\End_k(C)\right)^P\\
 &\cong\End_k(T)^P=\End_{kP}(T).
 \end{aligned}
\end{equation}
Choose a primitive idempotent in $\End_{kP}(T)$ whose image is isomorphic to $M$. Under \eqref{eq:homotopy-endomorphism-algebra}, this idempotent determines a decomposition $C\simeq X\oplus X'$ such that $H_0(X)\cong M$. By the direct summand argument for endosplit resolutions, $X$ is an endosplit $p$-permutation resolution of $M$; see \cite[Propositions 7.11.2 and 7.11.6 (iii)]{Lin18b}.

Let $\alpha\in E$. By Lemma \ref{lem:ambient-resolution}, ${}^\alpha C\cong C$, so ${}^\alpha X$ can be regarded as a direct summand of $C$ in $K^b(kP)$. Its homology is ${}^\alpha M\cong M$. The primitive idempotents of $\End_{kP}(T)$ whose images are isomorphic to $M$ form a single conjugacy class. Therefore, by the algebra isomorphism \eqref{eq:homotopy-endomorphism-algebra}, we have ${}^\alpha X\simeq X$. Finally, delete all nonzero contractible direct summands of $X$. The resulting complex is still an endosplit $p$-permutation resolution of $M$, and we replace $X$ by it. Neither $X$ nor ${}^\alpha X$ has a nonzero contractible direct summand. In a Krull--Schmidt category, a homotopy equivalence between two bounded complexes with this property is an isomorphism of complexes. Hence ${}^\alpha X\cong X$ for every $\alpha\in E$.
\end{proof}

\begin{proof}[Alternative proof of Theorem \ref{theo:main} for abelian $p$-groups]
Assume that $P$ is abelian, and keep the notation and hypotheses of Theorem \ref{theo:main}. Set $E=\Aut_\F(P)$. Since $E$ is closed under inverses, the $\F$-stability of $V$ implies that $V$ is $E$-stable with the convention at the beginning of this section. By Proposition \ref{prop:abelian-stable-resolution}, $V$ has an $E$-stable endosplit $p$-permutation resolution $X_V$.

Let $Q\leq P$ and let $\varphi\in\Hom_\F(Q,P)$. Since $P$ is abelian, every subgroup of $P$ is fully $\F$-centralised. Moreover, the subgroup $N_\varphi$ occurring in the extension axiom is equal to $P$, since conjugation by every element of $P$ is the identity. The extension axiom therefore extends $\varphi$ to an element $\alpha\in\Aut_\F(P)=E$. By the convention at the beginning of this section and the $E$-stability of $X_V$, we have ${}_\varphi X_V=\Res_Q^P({}^{\alpha^{-1}}X_V)\cong\Res_Q^P(X_V)$. Thus $X_V$ is $\F$-stable in the sense of Definition \ref{def:fusion-stability}, and hence it is the required endosplit $p$-permutation resolution of $V$.
\end{proof}

\begin{remark}
{\rm The proof above uses only the extension axiom for $\F$, and not the Sylow axiom. Hence the same proof applies to any fusion category on an abelian $p$-group which satisfies the extension axiom. In particular, the block fusion category associated with a maximal Brauer pair satisfies the extension axiom over any field of characteristic $p$. Indeed, in the proof of \cite[Theorem 8.5.2]{Lin18b}, the splitting field hypothesis is used only to verify the Sylow axiom; the proof of the extension axiom does not require this hypothesis. }
\end{remark}

\bigskip\noindent\textbf{Acknowledgements.}\quad The author would like to thank Prof. Markus Linckelmann for some helpful discussions and thank City St George's, University of London for its comfortable working environment. The author was supported by NSFC (12501024, 12471016), China Postdoctoral Science Foundation (GZC20252006, 2025T001HB) and China Scholarship Council (202506770066).

\bigskip
{\footnotesize School of Mathematics and Statistics, Central China Normal University, Wuhan 430079, China 
	
	Email address: xinhuang@mails.ccnu.edu.cn}

\end{document}